\documentclass[11pt]{amsart}
\usepackage{amsmath,amssymb,latexsym,esint,cite,mathrsfs}
\usepackage{verbatim,wasysym}
\usepackage[left=2.6cm,right=2.6cm,top=2.8cm,bottom=2.8cm]{geometry}
\usepackage[active]{srcltx}
\usepackage{microtype}
\usepackage{color,enumitem,graphicx}
\usepackage[colorlinks=true,urlcolor=blue, citecolor=red,linkcolor=blue,
linktocpage,pdfpagelabels, bookmarksnumbered,bookmarksopen]{hyperref}
\usepackage[hyperpageref]{backref}
\usepackage[english]{babel}

\newcommand{\e }{\varepsilon }
\newcommand{\R }{\mathbb{R} }

\newcommand{\pnorm}[2][]{\if #1'' \|#2\|_p \else \|#2\|_{#1} \fi}

\newtheorem{lemma}{Lemma}[section]

\newtheorem{proposition}[lemma]{Proposition}
\newtheorem{theorem}[lemma]{Theorem}

\theoremstyle{definition}

\newtheorem{remark}[lemma]{Remark}
\newtheorem{definition}[lemma]{Definition}

\numberwithin{equation}{section}

\title[Nonlocal problems with singular nonlinearity]{Nonlocal problems with singular nonlinearity}

\author[A.\ Canino]{Annamaria Canino}
\author[L.\ Montoro]{Luigi Montoro}
\author[B.\ Sciunzi]{Berardino Sciunzi}
\author[M.\ Squassina]{Marco Squassina}

\address[A.\ Canino, L.\ Montoro, B. Sciunzi]{Dipartimento di Matematica e Informatica
\newline\indent
Universit\`a della Calabria
\newline\indent
Via Pietro Bucci, Cubo 30B
Rende (CS), Italy}
\email{canino@mat.unical.it,montoro@mat.unical.it,sciunzi@mat.unical.it}

\address[M.\ Squassina]{Dipartimento di Informatica
	\newline\indent
	Universit\`a degli Studi di Verona,
	Verona, Italy}
\email{marco.squassina@univr.it}

\thanks{The authors are members of the Gruppo Nazionale per l'Analisi Matematica, la Probabilit\`a
	e le loro Applicazioni (GNAMPA) of the Istituto Nazionale di Alta Matematica (INdAM)}

\begin{document}

\begin{abstract}
We investigate existence and uniqueness of solutions for a class of nonlinear nonlocal problems involving
the fractional $p$-Laplacian operator and  singular nonlinearities.
\end{abstract}

\maketitle

\begin{center}
	\begin{minipage}{8.5cm}
		\small
		\tableofcontents
	\end{minipage}
\end{center}

\medskip

\section{Introduction}

\noindent
Let $p\in (1,\infty)$ and $s\in (0,1)$.
Let $\Omega\subset {\mathbb R}^N$ be a smooth bounded domain with $N>sp$.
We shall consider the following nonlocal quasilinear singular problem
\begin{equation}
\label{prob}
\begin{cases}
(- \Delta)_p^s\, u =\frac{f(x)}{u^\gamma} & \text{in $\Omega$},  \\
u>0 & \text{in $\Omega$},  \\
u  = 0  & \text{in $\mathbb{R}^N \setminus \Omega$}.
\end{cases}
\end{equation}
where $(-\Delta)_p^s$ is the fractional $p$-Laplacian operator, formally defined by
\begin{equation*}
(- \Delta)^s_p\, u(x):= 2\, \lim_{\varepsilon \searrow 0} \int_{\mathbb{R}^N \setminus B_\varepsilon(x)}
\frac{|u(x) - u(y)|^{p-2}\, (u(x) - u(y))}{|x - y|^{N+s\,p}}\, dy, \qquad x \in \mathbb{R}^N.
\end{equation*}
The aim of the paper is to prove the existence and the uniqueness of the solution to \eqref{prob}.
\vskip4pt

Let us first discuss the semilinear local case, $s=1$, $p=2$. In this setting the study of singular elliptic equations goes back to
 the pioneering work \cite{crandall}. Avoiding to  disclose the discussion,  we only mention here the contributions in \cite{boccardo,C23,CanDeg,CGS,gatica,saccon,kaw,lair,LM,stuart}, that settled up the issue in the semi-linear local case.
 The reader could be interested in obsarving that, in this case, by a simple change of variables it follows that the problem  is also related to problems involving a first order term of the type $\frac{|\nabla u|^2}{u}$. We refer the readers to \cite{arcoya,brandolini,giachetti} for related results in this setting.
To deal with singular problems, we have to face the fact that solutions are not in general in classical Sobolev spaces, because of the lack of regularity near the boundary. If already we consider the semi-linear local case it has been shown in \cite{LM}
that the solution cannot belong to $H^1_0(\Omega)$ if $\gamma\geqslant 3$.
Let us now state our result. Note that, due to the lack of regularity of the solutions near the boundary, the notion of solution has to be understood in the weak distributional meaning, for test functions compactly supported in the domain. Furthermore, the nonlocal nature of the operator has to be taken into account. Having this remarks in mind, the basic definition of solution can be formulated in the following:
\begin{definition}
	A function $u \in W^{s,p}_{{\rm loc}}(\Omega)\cap L^{p-1}(\Omega)$ is a weak solution to problem \eqref{prob} if
	$$
	u^{\max\{\frac{\gamma+p-1}{p},1\}}\in W^{s,p}_0(\Omega),\quad\,\,\, \frac{f(x)}{u^\gamma} \in L^{1}_{{\rm loc}}(\Omega),
	$$
	and we have
	\begin{equation}
		\label{v-f}
		\int_{\mathbb{R}^{2N}} \frac{|u(x) - u(y)|^{p-2}\, (u(x) - u(y))\, (\varphi(x) - \varphi(y))}{|x - y|^{N+s\,p}}\, dx\, dy
		=\int_{\Omega} \frac{f(x)}{u^\gamma}\,\varphi\, dx,
	\end{equation}
	for every $\varphi \in C^\infty_c(\Omega)$.
\end{definition}
According to such definition we have:
\begin{theorem}[Existence]
	\label{main thm existence}
Let $0<\gamma\leq 1$ and assume that
\begin{equation*}
 f\in L^{m}(\Omega),\quad
m:=\frac{Np}{N(p-1)+sp+\gamma(N-sp)}.
\end{equation*}
Then \eqref{prob} has a weak solution $u\in W^{s,p}_0(\Omega)$ with  ${\rm{ess inf}}_K\, u>0$
	for any compact $K\Subset\Omega$.\\
\noindent If  $\gamma>1$ and  $f\in L^{1}(\Omega)$, then \eqref{prob} has a weak a solution $u \in W^{s,p}_{{\rm loc}}(\Omega)\cap L^{p-1}(\Omega)$
	such that $u^{(\gamma+p-1)/p}\in W^{s,p}_0(\Omega)$ and ${\rm{ess inf}}_K\, u>0$
	for any compact $K\Subset\Omega$.
\end{theorem}
Actually the proof of Theorem \ref {main thm existence} will be carried out considering first  the simplest case $0<\gamma\leq 1$ (see Theorem \ref{thm:exgamma<=1}) and then the case $\gamma>1$ (see Theorem \ref{thm:exgamma>1}).
The proof relies on the well established technique introduce in \cite{boccardo}. Actually, via Shauder fixed point theorem, we find a solution to a regularized problem and then we perform a-priori uniform estimates to pass to the limit. Such a procedure has been investigated in the  nonlocal case for  $p=2$ in
 \cite{bbmp}, where a slightly different very weak notion of solution is considered that is allowed by the fact that the operator is linear and admits a double integration by parts. Since this is not the case for $p\neq 2$, we need a different approach which is rather technical and will be clear to the reader while reading the paper.

\vskip2pt
\noindent
Let us now turn to the uniqueness of the solution.
Since the way of understanding the boundary condition is not unambiguous, we start with the following:

\begin{definition}\label{rediri}
Let $u$ be such that $u  = 0$ in    $\mathbb{R}^N \setminus \Omega$.
We say
that $u\leq 0$ on $\partial\Omega$ if, for every $\e>0$, it follows that
\[
(u-\e)^+\in W^{s,p}_0(\Omega)\,.
\]
We will say that $u= 0$ on $\partial\Omega$ if $u$ is nonnegative  and $u\leq 0$ on $\partial\Omega$.
\end{definition}
Adopting such definition,  we will prove the following uniqueness result:
\begin{theorem}[Uniqueness]
	\label{main}
Let $\gamma>0$ and let $f\in L^1(\Omega)$ be  non-negative.
Then, under zero Dirichlet boundary conditions in the sense of Definition \ref{rediri}, the solution to \eqref{prob} is unique.
\end{theorem}

\noindent
It is worth emphasizing that  the proof of Theorem \ref{main} will follow via a more general comparison principle.\\

\noindent Having in mind Theorem \ref{main thm existence}, one may say that $u$ has zero Dirichlet boundary datum if
\begin{equation}\label{defsol2}
u^{\max\{\frac{\gamma+p-1}{p},1\}}\in W^{s,p}_0(\Omega).
\end{equation}
Our result applies in this case too since we have the following
\begin{proposition}\label{mainappos}
Let $\gamma>0$ and let $u$ be  non-negative with $u^{\max\{\frac{\gamma+p-1}{p},1\}}\in W^{s,p}_0(\Omega).$
Then $u$ fulfills zero  Dirichlet boundary conditions in the sense of Definition \ref{rediri}.
\end{proposition}
\begin{proof}We only need to prove the result in the case $\gamma>1$.
For $\e>0$ let us set
\[
\mathcal S_\e\,:={\rm supp} \,(u-\e)^+\qquad \mathcal Q_\e\,:=\,\mathbb{R}^{2N}\setminus \big(\mathcal S_\e^c\times\mathcal S_\e^c\big)\,.
\]
By Lemma \ref{vicesindaco} with $q=\frac{\gamma+p-1}{p}$, we have that
\begin{equation}\nonumber
\big|u^{\frac{\gamma+p-1}{p}}(x)-u^{\frac{\gamma+p-1}{p}}(y)\big|^p\,\geq\, \e^{\gamma-1} |u(x)-u(y)|^p\qquad\text{in $\mathcal Q_\e$},
\end{equation}
since either $u(x)\geq\e$ in $\mathcal Q_\e$ or $u(y)\geq\e$ in $\mathcal Q_\e$. From this we easily infer that
\begin{equation}\nonumber
\begin{split}
\int_{\mathbb{R}^{2N}}\frac{|(u-\e)^+(x)-(u-\e)^+(y)|^p}{|x - y|^{N+s\,p}}\,dx\,dy&\leq
\int_{\mathcal Q_\e}\frac{|u(x)-u(y)|^p}{|x - y|^{N+s\,p}}\,dx\,dy\\
&\leq \e^{1-\gamma}
\int_{\mathbb{R}^{2N}}\frac{|u^{\frac{\gamma+p-1}{p}}(x)-u^{\frac{\gamma+p-1}{p}}(y)|^p}{|x - y|^{N+s\,p}}\,dx\,dy,
\end{split}
\end{equation}
which concludes the proof.
\end{proof}

\noindent
By Theorem \ref{main}, thanks to Proposition \ref{mainappos}, we  deduce in fact a more general uniqueness result:
\begin{theorem}[Uniqueness]
	\label{main2}
Let $\gamma>0$ and let $f\in L^1(\Omega)$ be  non-negative and  let $u, v$ be weak solutions to \eqref{prob}.
Assume that $u$ and $v$ have zero Dirichlet boundary datum either in the sense of Definition \ref{rediri} or in the sense expressed in \eqref{defsol2}.
Then $u\equiv v$.
\end{theorem}
The technique used in the proof of the uniqueness result goes back to \cite{CanDeg} where the uniqueness of the solution is implicitly proved in the case
$f=1$, $p=2$ and $s=1$. Such a technique
was already improved in \cite{ccmcan} in the local semilinear case. In this setting it is worth mentioning the recent result in \cite{oliva} where singular problems with measure data are considered.
The local quasilinear case $p\neq 2$ was considered with a different technique in \cite{nodt}. For the nonlocal case we mention a related uniqueness result in \cite{fang}, where the case $f=1$ and $p=2$ is considered among other problems.\\

We end the introduction pointing out a first simple consequence of the uniqueness result:
\begin{theorem}[Symmetry]
	\label{Symmetry}
Let $u$ be the solution to \eqref{prob} under zero Dirichlet boundary condition.
Assume that the domain $\Omega$ is symmetric with respect to some hyperplane
$$
T_\lambda^\nu\,:=\,\{x\cdot\nu=\lambda\},\quad \lambda\in\R,\quad \nu\in S^{N-1}.
$$
Then, if $f$ is symmetric with respect to the hyperplane $T_\lambda^\nu$, then $u$ is symmetric with respect to the hyperplane $T_\lambda^\nu$ too.
In particular, if $\Omega$ is a ball or an annulus centered at the origin and $f$ is radially symmetric, then $u$ is radially symmetric.
\end{theorem}

\noindent

\section{Approximations}

We denote by $B_r(x_0)$ the $N-$dimensional open ball of radius $r$, centered at a point $x_0\in\mathbb{R}^N$.
The symbol $\|\cdot\|_{L^p(\Omega)}$ stands for the standard norm for the $L^p(\Omega)$ space. For a measurable function $u:\mathbb{R}^N\to\mathbb{R}$, we let
\[
[u]_{D^{s,p}(\mathbb{R}^N)} := \left(\int_{\mathbb{R}^{2N}} \frac{|u(x) - u(y)|^p}{|x - y|^{N+sp}}\, dx dy\right)^{1/p}
\]
be its Gagliardo seminorm. We consider the space
$$
W^{s,p}(\mathbb{R}^N) := \big\{u \in L^{p}(\mathbb{R}^N) : [u]_{D^{s,p}(\mathbb{R}^N)}<\infty\big\},
$$
endowed with norm $\|\,\cdot\,\|_{L^p(\mathbb{R}^N)}+[\,\cdot\,]_{D^{s,p}(\mathbb{R}^N)}$. If $\Omega\subset\mathbb{R}^N$ is an open set, not necessarily bounded, we consider
\[
W^{s,p}_0(\Omega) := \big\{u \in W^{s,p}(\mathbb{R}^N): \text{$u=0$ in $\mathbb{R}^N\setminus\Omega$}\big\},
\]
endowed with norm $[\,\cdot\,]_{D^{s,p}(\mathbb{R}^N)}$.
The imbedding $W^{s,p}_0(\Omega)\hookrightarrow L^r(\Omega)$ is continuous
for $1\le r \le p_s^\ast$ and compact for $1\le r< p_s^\ast$, where  $p^*_s:={N\,p}/{(N-s\,p)}$.
The space $W^{s,p}_0(\Omega)$ can be equivalently defined as the completion
of $C^\infty_0(\Omega)$ in the norm $[\,\cdot\,]_{D^{s,p}(\mathbb{R}^N)}$, provided $\partial\Omega$ is smooth enough.
We shall denote the localized Gagliardo seminorm by
$$
[u]_{W^{s,p}(\Omega)}:
=\left(\int_{\Omega\times\Omega} \frac{|u(x)-u(y)|^p}{|x-y|^{N+s\,p}}\,dx\,dy\right)^{1/p}.
$$
Finally define the  space
$$
W^{s,p}_{{\rm loc}}(\Omega) := \big\{u \in L^{p}(K) : [u]_{W^{s,p}(K)}<\infty\big\}, \quad\text{for all $K\Subset\Omega$}.
$$
\noindent
We first state a lemma dealing with the existence and uniqueness of solutions to $(- \Delta)_p^s\, u =f$.

\begin{lemma}
	\label{well-p}
	Let $f \in L^{\infty}(\Omega)$ and  $f\geq 0$, $f\not \equiv 0$. Then the problem
\begin{equation}\label{prob-fbounded}
\begin{cases}
(- \Delta)_p^s\, u =f& \text{in $\Omega$},  \\
u>0 & \text{in $\Omega$},  \\
u  = 0  & \text{in $\mathbb{R}^N \setminus \Omega$,}
\end{cases}
\end{equation}
admits a unique solution  $u \in W^{s,p}_0(\Omega)$.
\end{lemma}
\begin{proof}
To prove the existence of a solution to \eqref{prob-fbounded}, we minimize the functional
\begin{equation*}
\mathcal J(u)=\frac{1}{p}\int_{\mathbb{R}^{2N}}\frac{|u(x)-u(y)|^p}{|x-y|^{N+ps}}\,dx\,dy-\int_{\Omega}f(x)\,u\,dx, \qquad u\in W^{s,p}_0(\Omega),
\end{equation*}
and then we look for a solution to \eqref{prob-fbounded} as a critical point of $\mathcal J (u)$. In fact, we have that
\begin{itemize}
\item[${\rm (i)}$] $\mathcal J(u)$ is coercive, since  by the Sobolev embedding it follows
\begin{equation}\nonumber
\mathcal J(u)\geq \frac 1p \int_{\mathbb{R}^{2N}}\frac{|u(x)-u(y)|^p}{|x-y|^{N+ps}}\,dx\,dy-
C\|f\|_{L^{\infty}(\Omega)}\left(\int_{\mathbb{R}^N}\frac{|u(x)-u(y)|^p}{|x-y|^{N+ps}}\,dx\,dy\right)^{\frac 1p}.
\end{equation}
\item[${\rm (ii)}$] $\mathcal J(u)$ is weakly lower semi-continuous in   $W^{s,p}_0(\Omega)$.
\end{itemize}
Then choosing a non negative  minimizing  sequence $\{u_n\}_{n\in{\mathbb N}}$ (and since $f\geq 0$ it is not restrictive to assume that  $u_n(x)\geq0$ a.e. in $\mathbb{R}^N$, if not take $\{|u_n(x)|\}_{n\in{\mathbb N}}$), the existence of a minimum of $\mathcal J$
and thus of a nonnegative solution $u$ to \eqref{prob-fbounded},  follows by a standard minimization procedure.
That $u>0$ follows by the strong maximum principle stated in \cite[Theorem A.1]{BF}.
We show now that the solution to problem \eqref{prob-fbounded} is unique.
Let us suppose that $u_1,u_2\in W^{s,p}_0(\Omega)$ are weak solutions to \eqref{prob-fbounded}. Therefore,
for all $\varphi\in W^{s,p}_0(\Omega)$ we have
\begin{eqnarray*}
&&\int_{\mathbb{R}^{2N}} \frac{|u_1(x) - u_1(y)|^{p-2}\, (u_1(x) - u_1(y))\, (\varphi(x) - \varphi(y))}{|x - y|^{N+s\,p}}\, dx\, dy
=\int_{\Omega}f(x)\,\varphi\, dx,\\
&&\int_{\mathbb{R}^{2N}} \frac{|u_2(x) - u_2(y)|^{p-2}\, (u_2(x) - u_2(y))\, (\varphi(x) - \varphi(y))}{|x - y|^{N+s\,p}}\, dx\, dy
=\int_{\Omega}f(x)\,\varphi\, dx.
\end{eqnarray*}
Subtracting the two equations yields
\begin{equation*}
\int_{\mathbb{R}^{2N}} \frac{\big (|u_1(x) - u_1(y)|^{p-2}\, (u_1(x) - u_1(y))-|u_2(x) - u_2(y)|^{p-2}\, (u_2(x) - u_2(y))\big)\, (\varphi(x) - \varphi(y))}{|x - y|^{N+s\,p}}=0.
\end{equation*}
Inserting $\varphi(x)=w(x):=u_1(x)-u_2(x)$, using  elementary inequalities (cf.\ \cite[Section 10]{peter}), yields
\begin{equation*}
\int_{\mathbb{R}^{2N}}\frac{(|u_1(x) - u_1(y)|+|u_2(x) - u_2(y)|)^{p-2}|w(x)-w(y)|^{2}}{|x - y|^{N+s\,p}}\, dx dy\leq 0,\quad\,\,\, \text{if $1<p<2$},
\end{equation*}
as well as
\begin{equation*}
\int_{\mathbb{R}^{2N}}\frac{|w(x)-w(y)|^{p}}{|x - y|^{N+s\,p}}\, dx dy\leq 0,\quad\,\,\, \text{if $p\geq  2$}.
\end{equation*}
In both cases the inequalities yield $w(x)=C$ for all $x\in {\mathbb R}^N$
and some constant $C\in {\mathbb R}$. Since $u_i=0$ on $\Omega^c$, we get $w=0$ on $\Omega^c$.
Therefore $C=0$ and the assertion follows.
\end{proof}

\noindent
Solutions of the problem $(- \Delta)_p^s u = f(x)$ enjoy the useful $L^q$-estimate (cf.\ \cite[Lemma 2.3]{MRmpsy}), that we state in the following 
\begin{lemma}[Summability lemma]
	\label{summab}
	Let $f \in L^q(\Omega)$ for some $1 < q \le \infty$ and assume that $u \in W^{s,p}_0(\Omega)$ is a
	weak solution of the equation $(- \Delta)_p^s\, u = f(x)$ in $\Omega$. Then
	\begin{equation*} 
	\pnorm[r]{u} \le C \pnorm[q]{f}^{1/(p-1)},
	\end{equation*}
	where
	\[
	r := \begin{cases}
	\dfrac{N\, (p - 1)\, q}{N - spq}, & 1 < q < \dfrac{N}{sp}, \\[10pt]
	\infty, & \dfrac{N}{sp} < q \le \infty,
	\end{cases}
	\]
	and $C = C(N,\Omega,p,s,q) > 0$.
\end{lemma}

\noindent
We consider, for a given $f\in L^1(\Omega)$, with $f\geq 0$ the truncation
$$
f_n(x):=\min\{f(x),n\},\quad\,\, x\in\Omega.
$$
Then, we consider the approximating problems
\begin{equation}
\tag{${\mathcal P}_n$}
\label{prob-approx}
\begin{cases}
(- \Delta)_p^s\, u_n =\frac{f_n(x)}{(u_n+1/n)^\gamma} & \text{in $\Omega$},  \\
u_n>0 & \text{in $\Omega$},  \\
u_n  = 0  & \text{in $\mathbb{R}^N \setminus \Omega$}.
\end{cases}
\end{equation}

\begin{proposition}\label{pro:solapprox}
For any $n\geq 1$ there exists a weak solution $u_n\in W^{s,p}_0(\Omega)\cap L^\infty(\Omega)$ to \eqref{prob-approx}.
\end{proposition}
\begin{proof}
Given $n\in {\mathbb N}$ and a function $u\in L^p(\Omega)$, in light of Lemma~\ref{well-p} there exists a unique
solution  $w\in W^{s,p}_0(\Omega)$ to the problem
\begin{equation}\label{prob-fbounded-choi}
\begin{cases}
(- \Delta)_p^s\, w =\frac{f_n(x)}{(u^++1/n)^\gamma}   & \text{in $\Omega$},  \\
w>0 & \text{in $\Omega$},  \\
w  = 0  & \text{in $\mathbb{R}^N \setminus \Omega$,}
\end{cases}
\end{equation}
where $u^+:=\max\{u,0\}.$ Therefore we may define the map $L^p(\Omega)\ni u\mapsto w:=S(u) \in W^{s,p}_0(\Omega)\subset L^p(\Omega)$, where
$w$ is the unique solution to \eqref{prob-fbounded-choi}. Using $w$ as test function in \eqref{prob-fbounded-choi} we obtain
\begin{equation*}
\int_{\mathbb{R}^{2N}} \frac{|w(x) - w(y)|^p}{|x - y|^{N+s\,p}}\, dx\, dy
=\int_{\Omega} \frac{f_n(x)}{(u^++1/n)^\gamma} \,w\, dx\leq n^{\gamma+1}\|w\|_{L^1(\Omega)},
\end{equation*}
and thus by, Sobolev imbedding, we have that
\begin{equation}
\label{estim}
\left(\int_{\mathbb{R}^{2N}} \frac{|w(x) - w(y)|^p}{|x - y|^{N+s\,p}}\, dx\, dy\right)^{\frac{1}{p}}
\leq Cn^{\frac{\gamma+1}{p-1}},
\end{equation}
for some $C=C(p,s,N,\Omega)$ (independent of $u$), so that the ball of radius $R:=Cn^{\frac{\gamma+1}{p-1}}$ in $W^{s,p}_0(\Omega)$, is invariant under the action of  $S$. Now, in order to apply the Schauder's fixed point theorem to $S$ and then to obtain a solution to \eqref{prob-fbounded-choi}, we have to prove the continuity and the compactness of $S$ as  an operator from $W^{s,p}_0(\Omega)$ to $W^{s,p}_0(\Omega)$.
\vskip3pt
\noindent
$\bullet$ ({\em Continuity of $S$}). Denoting $w_k:=S(u_k)$ and $w:=S(u)$, then
\begin{equation*}
\lim_{k\to\infty}\|w_k-w\|_{W^{s,p}_0(\Omega)}=0,\,\,\,\quad\text{if}\,\,\, \lim_{k\to\infty}\|u_k-u\|_{W^{s,p}_0(\Omega)}=0.
\end{equation*}
By the strong convergence of $\{u_k\}_{k\in{\mathbb N}}$ in ${W^{s,p}_0(\Omega)}$, up to a subsequence,
we have $u_k \to u$ in $L^{p_s^*}(\Omega)$ and $u_k\to u$ a.e.\ in $\Omega$ as $k\to\infty$.\
Considering the corresponding sequence of solutions $\{w_k\}_{k\in \mathbb{N}}$, arguing as in the proof of Lemma \ref{well-p}, setting $\bar w_k(x):= w_k(x)-w(x) $  we obtain
\begin{eqnarray}\label{eq:pmi2}
&&\int_{\mathbb{R}^{2N}}\frac{(|w_k(x) - w_k(y)|+|w(x) - w(y)|)^{p-2}|\bar w_k(x)-\bar w_k(y)|^{2}}{|x - y|^{N+s\,p}}\, dx dy\\\nonumber
&&\leq \int_{\Omega}\left( \frac{f_n(x)}{(u_k^++1/n)^\gamma}-\frac{f_n(x)}{(u^++1/n)^\gamma} \right)\,(w_k-w)\, dx,\quad\,\,\, \text{if $1<p<2$},
\end{eqnarray}
as well as
\begin{eqnarray}\label{eq:pma2}
&&\int_{\mathbb{R}^{2N}}\frac{|\bar w_k(x)-\bar w_k(y)|^{p}}{|x - y|^{N+s\,p}}\, dx dy
\\\nonumber &&\leq \int_{\Omega}\left( \frac{f_n(x)}{(u_k^++1/n)^\gamma}-\frac{f_n(x)}{(u^++1/n)^\gamma} \right)\,(w_k-w)\, dx,\quad\,\,\, \text{if $p\geq  2$}.
\end{eqnarray}
Let us consider the right-hand side of \eqref{eq:pma2}.  Using H\"older and Sobolev inequalities we infer that
\begin{eqnarray}\nonumber
&&\left|\int_{\Omega}\left(\frac{f_n(x)}{(u_k^++1/n)^\gamma}-\frac{f_n(x)}{(u^++1/n)^\gamma} \right)\,(w_k-w)\, dx\right|\\\nonumber&&\leq\left(\int_{\Omega}\left|\frac{f_n(x)}{(u_k^++1/n)^\gamma}-\frac{f_n(x)}{(u^++1/n)^\gamma} \right| ^{(p^*_s)'}\,dx\right)^\frac{1}{(p^*_s)'}\|\bar w_k\|_{L^{p_s^*}(\Omega)}\\\nonumber
&&\leq C\left(\int_{\Omega}\left|\frac{f_n(x)}{(u_k^++1/n)^\gamma}-\frac{f_n(x)}{(u^++1/n)^\gamma} \right| ^{(p^*_s)'}\, dx\right)^\frac{1}{(p^*_s)'}\|\bar w_k\|_{W^{s,p}_0(\Omega)},
\end{eqnarray}
where $(p^*_s)'=Np/(N(p-1)+sp)$ and $C=C(p,s,N)$ is a positive constant. From \eqref{eq:pma2} we get
\begin{eqnarray}\label{eq:pma22}
&&\|w_k-w\|_{W^{s,p}_0(\Omega)}
\\\nonumber &&\leq  C\left(\int_{\Omega}\left|\frac{f_n(x)}{(u_k^++1/n)^\gamma}-\frac{f_n(x)}{(u^++1/n)^\gamma} \right|^{(p^*_s)'}\, dx\right)^\frac{1}{(p-1)\cdot(p^*_s)'}\quad\,\,\, \text{if $p\geq  2$}.
\end{eqnarray}
Observing that
\begin{equation}\label{eq:nnnn}
\left|\frac{f_n(x)}{(u_k^++1/n)^\gamma}-\frac{f_n(x)}{(u^++1/n)^\gamma}\right|\leq 2 n^{\gamma +1},
\end{equation}
by the dominated convergence theorem and by the fact that $u_k(x)\to u(x)$ a.e., from  \eqref{eq:pma22} we conclude that
$$
\lim_{k\rightarrow +\infty}\|w_k-w\|_{W^{s,p}_0(\Omega)}=0,
$$
showing, in the case $p\geq 2$, that the operator $S$ is continuous from  ${W^{s,p}_0(\Omega)}$ to ${W^{s,p}_0(\Omega)}$. From~\eqref{eq:pmi2}, a similar argument, shows the continuity of $S$ from  ${W^{s,p}_0(\Omega)}$ to ${W^{s,p}_0(\Omega)}$ for $1<p<2$.
\vskip4pt
\noindent
$\bullet$ ({\em Compactness of $S$}). \ Let $\{u_k\}_{k\in \mathbb{N}}\subset W^{s,p}_0(\Omega)$. Denoting $w_k:=S(u_k)$, we  show that, up to a subsequence and for some $w\in {W^{s,p}_0(\Omega)}$, it holds
\begin{equation*}
\lim_{k\rightarrow \infty}\|w_k-w\|_{W^{s,p}_0(\Omega)}=0.
\end{equation*}
Let $\{u_k\}_{k\in \mathbb{N}}\subset W^{s,p}_0(\Omega)$ with $\|u_k\|_{{W^{s,p}_0(\Omega)}}\leq C$ for all $k\geq 1$.\
Then, up to a subsequence, we have $u_k \rightharpoonup u$ in $W^{s,p}_0(\Omega),$ as well as
$u_k \to u$ in $L^r(\Omega)$, for $1\le r< p_s^\ast$.
In view of \eqref{estim}, we have $\|S(u_k)\|_{{W^{s,p}_0(\Omega)}}\leq C$ for some constant $C$
independent of $k$, and therefore
\begin{equation*}
S(u_k) \rightharpoonup  w, \,\,\, \text{in $W^{s,p}_0(\Omega)$},\qquad
S(u_k) \to  w, \,\,\, \text{in $L^r(\Omega)$, \,\,for $1\le r< p_s^\ast$},
\end{equation*}
for some $w\in W^{s,p}_0(\Omega)$.  Then for all $\varphi \in W^{s,p}_0(\Omega)$
\begin{equation}
\label{approx}
\int_{\mathbb{R}^{2N}} \frac{|w_k(x) - w_k(y)|^{p-2}\, (w_k(x) - w_k(y))\, (\varphi(x) - \varphi(y))}{|x - y|^{N+s\,p}}\, dx\, dy
=\int_{\Omega} \frac{f_n(x)}{(u^+_k+1/n)^\gamma} \,\varphi\, dx.
\end{equation}
We show now that, letting $k$ to infinity,   \eqref{approx} converges to
\begin{equation}
\label{limite}
\int_{\mathbb{R}^{2N}} \frac{|w(x) - w(y)|^{p-2}\, (w(x) - w(y))\, (\varphi(x) - \varphi(y))}{|x - y|^{N+s\,p}}\, dx\, dy
=\int_{\Omega} \frac{f_n(x)}{(u^++1/n)^\gamma} \,\varphi\, dx.
\end{equation}
By the dominated convergence theorem, it is readily seen that
$$
\lim_{k\to\infty}\int_{\Omega} \frac{f_n(x)}{(u^+_k+1/n)^\gamma} \,\varphi\, dx=\int_{\Omega} \frac{f_n(x)}{(u^++1/n)^\gamma} \,\varphi\, dx.
$$
Furthermore, since the sequence
$$
\left\{\frac{|w_k(x) - w_k(y)|^{p-2}\, (w_k(x) - w_k(y))}{|x - y|^{\frac{N+s\,p}{p'}}}\right\}_{k\in {\mathbb N}}\quad \text{is bounded in $L^{p'}({\mathbb R}^{2N})$},
$$
and by the pointwise convergence of $w_k(x)$ to $w(x)$
$$
\frac{|w_k(x) - w_k(y)|^{p-2}\, (w_k(x) - w_k(y))}{|x - y|^{\frac{N+s\,p}{p'}}}\to \frac{|w(x) - w(y)|^{p-2}\, (w(x) - w(y))}{|x - y|^{\frac{N+s\,p}{p'}}}
\quad \text{a.e.\ in ${\mathbb R}^{2N}$},
$$
it follows by standard results that, up to a subsequence,
$$
\frac{|w_k(x) - w_k(y)|^{p-2}\, (w_k(x) - w_k(y))}{|x - y|^{\frac{N+s\,p}{p'}}}
\rightharpoonup \frac{|w(x) - w(y)|^{p-2}\, (w(x) - w(y))}{|x - y|^{\frac{N+s\,p}{p'}}}
\quad \text{weakly in $L^{p'}({\mathbb R}^{2N})$}.
$$
Then, since
$$
\frac{\varphi(x) - \varphi(y)}{|x - y|^{\frac{N+s\,p}{p}}}\in L^p({\mathbb R}^{2N}),
$$
we conclude that the l.h.s.\ of \eqref{approx} converges to the l.h.s.\ of \eqref{limite}. Whence, \eqref{limite} holds, that is, in particular, $w=S(u)$.
Arguing as for \eqref{eq:pmi2} and~\eqref{eq:pma2} setting $w_k(x)=S(u_k)$ and $\bar w_k(x):= w_k(x)-w(x)$, we infer that
\begin{eqnarray}\nonumber
&&\int_{\mathbb{R}^{2N}}\frac{(|w_k(x) - w_k(y)|+|w(x) - w(y)|)^{p-2}|\bar w_k(x)-\bar w_k(y)|^{2}}{|x - y|^{N+s\,p}}\, dx dy\\\nonumber
&&\leq \|w_k-w\|_{L^p(\Omega)} \left(\int_{\Omega}\left|\frac{f_n(x)}{(u_k^++1/n)^\gamma}-\frac{f_n(x)}{(u^++1/n)^\gamma} \right|^{p'}\, dx\right)^{\frac{1}{p'}},\quad\,\,\, \text{if $1<p<2$},
\end{eqnarray}
as well as
\begin{eqnarray}\nonumber
&&\int_{\mathbb{R}^{2N}}\frac{|\bar w_k(x)-\bar w_k(y)|^{p}}{|x - y|^{N+s\,p}}\, dx dy
\\\nonumber &&\leq \|w_k-w\|_{L^p(\Omega)} \left(\int_{\Omega}\left|\frac{f_n(x)}{(u_k^++1/n)^\gamma}-\frac{f_n(x)}{(u^++1/n)^\gamma} \right|^{p'}\, dx\right)^{\frac{1}{p'}},\quad\,\,\, \text{if $p\geq  2$},
\end{eqnarray}
where $p'=p/(p-1)$. Using \eqref{eq:nnnn}, the last two  equations imply that
\begin{equation*}
\lim_{k\rightarrow +\infty}\|S(u_k)-S(u)\|_{W^{s,p}_0(\Omega)}=0,
\end{equation*}
that is the compactness of $S$ from ${W^{s,p}_0(\Omega)}$ to ${W^{s,p}_0(\Omega)}$.
Schauder's fixed point theorem provides that existence of $u_n\in W^{s,p}_0(\Omega)$ such that $u_n=S(u_n)$, that is a weak solution to
\begin{equation}\label{pro:u_n}
\begin{cases}
(- \Delta)_p^s\, u_n =\frac{f_n(x)}{(u_n+1/n)^\gamma}   & \text{in $\Omega$},  \\
u_n>0 & \text{in $\Omega$},  \\
u_n  = 0  & \text{in $\mathbb{R}^N \setminus \Omega$.}
\end{cases}
\end{equation}
Since the r.h.s.\  of \eqref{pro:u_n} belongs to $L^{\infty}(\Omega)$,
by virtue of Lemma~\ref{summab} we have $u_n\in L^{\infty}(\Omega)$.
\end{proof}

\begin{lemma}[Monotonicity]\label{lem:monotonicity}
The sequence $\{u_n\}_{n\in {\mathbb N}}$ found in the previous lemma satisfies
$$
u_n(x)\leq u_{n+1}(x),\quad \text{for a.e.\ $x\in\Omega$},
$$
and
$$
\text{$u_n(x)\geq \sigma>0$,\quad
	for a.e. $x\in\omega\Subset\Omega$,}
$$	
for some positive constant $\sigma=\sigma(\omega)$.
\end{lemma}
\begin{proof}
We have, for any $n\in {\mathbb N}$, that for all $\varphi\in W^{s,p}_0(\Omega)$
\begin{equation*}
\int_{\mathbb{R}^{2N}} \frac{|u_n(x) - u_n(y)|^{p-2}\, (u_n(x) - u_n(y))\, (\varphi(x) - \varphi(y))}{|x - y|^{N+s\,p}}\, dx\, dy
=\int_{\Omega} \frac{f_n(x)}{(u_n+1/n)^\gamma} \,\varphi\, dx.
\end{equation*}
as well as for all $\varphi\in  W^{s,p}_0(\Omega)$
\begin{equation*}
\int_{\mathbb{R}^{2N}} \frac{|u_{n+1}(x) - u_{n+1}(y)|^{p-2}\, (u_{n+1}(x) - u_{n+1}(y))\, (\varphi(x) - \varphi(y))}{|x - y|^{N+s\,p}}\, dx\, dy
=\int_{\Omega} \frac{f_{n+1}(x)}{(u_{n+1}+1/(n+1))^\gamma} \,\varphi\, dx.
\end{equation*}
By taking $\varphi=w=(u_n-u_{n+1})^+\in W^{s,p}_0(\Omega)$ as test function in the  formula above
and subtracting the second from the first, concerning the r.h.s. (and recalling that $f_n\leq f_{n+1}$ a.e.) we get
\begin{align*}
&\int_{\Omega} \frac{f_n(x)}{(u_n+1/n)^\gamma} \,(u_n-u_{n+1})^+\, dx-\int_{\Omega} \frac{f_{n+1}(x)}{(u_{n+1}+1/(n+1))^\gamma} (u_n-u_{n+1})^+\, dx \\
&\leq \int_{\Omega} f_{n+1}(u_n-u_{n+1})^+
\frac{(u_{n+1}+1/(n+1))^\gamma-(u_n+1/n)^\gamma}{(u_n+1/n)^\gamma (u_{n+1}+1/(n+1))^\gamma}dx\leq 0.
\end{align*}
Then, if $I_p(s):=|s|^{p-2}s$, we conclude that
\begin{equation}
\label{in-neg}
\int_{\mathbb{R}^{2N}} \frac{\big (I_p(u_n(x) - u_n(y))-I_p(u_{n+1}(x) - u_{n+1}(y))\big)\,
	(w(x) - w(y))}{|x - y|^{N+s\,p}}\leq 0.
\end{equation}	
Now, arguing exactly as in the proof of \cite[Lemma 9]{lindglindq}, we get
$$
\big (I_p(u_n(x) - u_n(y))-I_p(u_{n+1}(x) - u_{n+1}(y))\big)\, (w(x) - w(y))\geq 0,\quad\text{for a.e. $(x,y)\in {\mathbb R}^{2N}$},
$$
with the  {\em strict} inequality, unless it holds
\begin{equation}
\label{concl}
(u_n(x)-u_{n+1}(x))^+=(u_n(y)-u_{n+1}(y))^+,\quad\text{for a.e. $(x,y)\in {\mathbb R}^{2N}$}.
\end{equation}
On the other hand, by \eqref{in-neg}, we have
$$
\big (I_p(u_n(x) - u_n(y))-I_p(u_{n+1}(x) - u_{n+1}(y))\big)\, (w(x) - w(y))=0,\quad\text{for a.e. $(x,y)\in {\mathbb R}^{2N}$}.
$$
Therefore, \eqref{concl} holds true, namely
\begin{equation*}
(u_n(x)-u_{n+1}(x))^+=C,\quad\text{for a.e. $(x,y)\in {\mathbb R}^{2N}$},
\end{equation*}
for some constant $C$. Since $u_n=u_{n+1}=0$ on ${\mathbb R}^{N}\setminus\Omega$ it follows that $C=0$, which implies in turn that
$u_n(x)\leq u_{n+1}(x)$, for a.e.\ $x\in\Omega$. This concludes the proof of the first assertion. Concerning the second assertion, we observe that
we know that $u_1\in L^\infty(\Omega)$, yielding
$$
(- \Delta)_p^s\, u_1 =\frac{f_1(x)}{(u_1+1)^\gamma} \in L^\infty(\Omega) .
$$
Then, by \cite[Theorem 1.1]{ibero} we deduce that $u_1\in C^{0,\alpha}(\bar{\Omega})$ for some $\alpha\in (0,1)$. In particular by the strong maximum principle,
$$
\text{$u_1(x)\geq \sigma>0$,\quad
	for a.e. $x\in\omega\subset\subset\Omega$}
$$	
and $\sigma=\sigma(\omega)$. The second assertion then follows by monotonicity.
\end{proof}

\section{Existence of solutions}

To prove the existence of a solution to \eqref{prob} we use the sequence of solutions $\{u_n\}_{n\in \mathbb N}$ of problem~\eqref{prob-approx}  (see Proposition \ref{pro:solapprox}) and then, using some a-priori estimates,  we pass to the limit.

\subsection{Existence in the case $0<\gamma\leq 1$}

First of all, we prove  the following

\begin{lemma}\label{lem:apiori-1}
Let $\{u_n\}_{n\in {\mathbb N}}\subset W^{s,p}_0(\Omega)\cap L^\infty(\Omega)$ be the sequence of
solution to problem  \eqref{prob-approx} provided by Proposition \ref{pro:solapprox}. Assume that
\begin{equation}\label{eq:mumerolemma3.1}
0<\gamma \leq 1,\qquad f\geq 0,\quad f\in L^{m}(\Omega),\quad
m:=\frac{Np}{N(p-1)+sp+\gamma(N-sp)}.
\end{equation}
Then $\{u_n\}_{n\in {\mathbb N}}$ is bounded in $W^{s,p}_0(\Omega)$.
\end{lemma}
\begin{proof}
In the case $0<\gamma<1$,  taking $u_n$ as test function in \eqref{prob-approx}, as $f_n\leq f$ we get
\begin{align}\label{eq:stimaun}
\int_{\mathbb{R}^{2N}}\frac{|u_n(x)-u_n(y)|^p}{|x-y|^{N+ps}}\,dx\,dy &\leq \int_{\Omega} \frac{f_n(x)}{(u_n+1/n)^\gamma} \,u_n\, dx\leq \int_{\Omega} f_n(x) \,u_n^{1-\gamma}\, dx\\\nonumber
&\leq\|f\|_{L^m(\Omega)} \left(\int_{\Omega}u_n^{(1-\gamma)m'}\, dx\right)^{\frac{1}{m'}},
\end{align}
where $m'=m/(m-1)$. Since $(1-\gamma)m'=p^*_s$, by the Sobolev embedding, we obtain
\begin{equation*}
\left(\int_{\Omega}u_n^{(1-\gamma)m'}\, dx\right)^{\frac{1}{m'}}\leq C \left(\int_{\mathbb{R}^{2N}}\frac{|u_n(x)-u_n(y)|^p}{|x-y|^{N+ps}}\,dx\,dy\right)^{\frac{p^*_s}{pm'}},
\end{equation*}
for some constant $C=C(p,s,N)>0$.\ Finally, since $p^*_s/(pm')<1$, from \eqref{eq:stimaun} we get
$$
\sup_{n\in {\mathbb N}}\|u_n\|_{W^{s,p}_0(\Omega)}\leq C(f,p,s,\gamma,N).
$$
If instead $\gamma =1$,  then arguing as for \eqref{eq:stimaun}, we get
\begin{equation*}
\int_{\mathbb{R}^{2N}}\frac{|u_n(x)-u_n(y)|^p}{|x-y|^{N+ps}}\,dx\,dy\leq \int_{\Omega} \frac{f_n(x)}{u_n+1/n} \,u_n\, dx\leq \int_{\Omega} f(x)  dx,
\end{equation*}
which yields again the desired boundedness.
\end{proof}

\begin{theorem}[Existence, $0<\gamma\leq 1$]
	\label{thm:exgamma<=1}
Assume that \eqref{eq:mumerolemma3.1} holds.
Then  problem \eqref{prob} admits  a weak solution $u\in W^{s,p}_0(\Omega)$.
\end{theorem}
\begin{proof}
By virtue of Lemma \ref{lem:apiori-1}, the sequence  of solutions $\{u_n\}_{n\in \mathbb N}\subset W^{s,p}_0(\Omega)\cap L^\infty(\Omega)$ of problem~\eqref{prob-approx}
provided by Proposition \ref{pro:solapprox} is bounded in ${W^{s,p}_0(\Omega)}$.
Then, up to a subsequence, we have $u_n \rightharpoonup u$ in $W^{s,p}_0(\Omega),$
$u_n \to u$ in $L^r(\Omega)$ for $1\le r< p_s^\ast$ and $u_n \rightarrow u$ a.e.\ in $\Omega$
and, furthermore, by Lemma~\ref{lem:monotonicity}, we have
$$
\text{for all $K\Subset\Omega$ there exists $\sigma_K>0$ such that $u(x)\geq \sigma_K>0$,\quad
	for a.e. $x\in K$.}
$$	
We have
\begin{equation}\label{eq:equazPn}
\int_{\mathbb{R}^{2N}} \frac{|u_n(x) - u_n(y)|^{p-2}\, (u_n(x) - u_n(y))\, (\varphi(x) - \varphi(y))}{|x - y|^{N+s\,p}}\, dx\, dy
=\int_{\Omega} \frac{f_n(x)}{(u_n+1/n)^\gamma} \,\varphi\, dx,
\end{equation}
for all $\varphi \in C^\infty_c(\Omega)$.  Since the sequence
$$
\left\{\frac{|u_n(x) - u_n(y)|^{p-2}\, (u_n(x) - u_n(y))}{|x - y|^{\frac{N+s\,p}{p'}}}\right\}_{n\in {\mathbb N}}\quad \text{is bounded in $L^{p'}({\mathbb R}^{2N})$},
$$
and by the point-wise convergence of $u_n$ to $u$
$$
\frac{|u_n(x) - u_n(y)|^{p-2}\, (u_n(x) - u_n(y))}{|x - y|^{\frac{N+s\,p}{p'}}}\to \frac{|u(x) - u(y)|^{p-2}\, (u(x) - u(y))}{|x - y|^{\frac{N+s\,p}{p'}}}
\quad \text{a.e.\ in ${\mathbb R}^{2N}$},
$$
it follows by standard results that
$$
\frac{|w_k(x) - w_k(y)|^{p-2}\, (w_k(x) - w_k(y))}{|x - y|^{\frac{N+s\,p}{p'}}}
\rightharpoonup \frac{|w(x) - w(y)|^{p-2}\, (w(x) - w(y))}{|x - y|^{\frac{N+s\,p}{p'}}}
\quad \text{weakly in $L^{p'}({\mathbb R}^{2N})$}.
$$
Then, since for $\varphi \in C^\infty_c(\Omega)$ we have
$$
\frac{\varphi(x) - \varphi(y)}{|x - y|^{\frac{N+s\,p}{p}}}\in L^p({\mathbb R}^{2N}),
$$
we conclude that
\begin{eqnarray}\nonumber
&&\lim_{n \rightarrow +\infty}\int_{\mathbb{R}^{2N}} \frac{|u_n(x) - u_n(y)|^{p-2}\, (u_n(x) - u_n(y))\, (\varphi(x) - \varphi(y))}{|x - y|^{N+s\,p}}\, dx\, dy
=
\\\nonumber
&&\int_{\mathbb{R}^{2N}} \frac{|u(x) - u(y)|^{p-2}\, (u(x) - u(y))\, (\varphi(x) - \varphi(y))}{|x - y|^{N+s\,p}}\, dx\, dy,
\end{eqnarray}
for all $\varphi \in C^\infty_c(\Omega)$.
Concerning the right-hand side of formula \eqref{eq:equazPn}, recalling Lemma \ref{lem:monotonicity}, for any $\varphi \in C^\infty_c(\Omega)$
with ${\rm supp}(\varphi)=K$, there exists $\sigma_K>0$ independent of $n$ such that
\begin{equation*}
\left |\frac{f_n(x)\,\varphi}{(u_n+1/n)^\gamma} \right|\leq \sigma_K^\gamma |f(x)\varphi(x)| \in L^1(\Omega).
\end{equation*}
By the dominated convergence theorem we conclude that
$$
\lim_{n\to\infty}\int_{\Omega} \frac{f_n(x)\,\varphi}{(u_n+1/n)^\gamma} \,\varphi\, dx=\int_{\Omega} \frac{f(x)}{u^\gamma} \,\varphi\, dx.
$$
Finally, passing to the limit in \eqref{eq:equazPn}, we conclude that
\begin{equation*}
\int_{\mathbb{R}^{2N}} \frac{|u(x) - u(y)|^{p-2}\, (u(x) - u(y))\, (\varphi(x) - \varphi(y))}{|x - y|^{N+s\,p}}\, dx\, dy
=\int_{\Omega} \frac{f(x)}{u^\gamma}\,\varphi\, dx,
\end{equation*}
for all $\varphi \in C^\infty_c(\Omega)$, namely $u$ is a solution to \eqref{prob}.
\end{proof}

%
%

\subsection{Existence in the case $\gamma>1$}

\noindent
First, we recall the following result from \cite[Lemma 3.3]{BP}.

\begin{proposition}
	\label{convex}
Let $F\in L^q(\Omega)$ with $q>N/(sp)$ and let $u\in W^{s,p}_0(\Omega)\cap L^\infty(\Omega)$ be such that 	
\begin{equation*}
\int_{\mathbb{R}^{2N}} \frac{|u(x) - u(y)|^{p-2}\, (u(x) - u(y))\, (\varphi(x) - \varphi(y))}{|x - y|^{N+s\,p}}\, dx\, dy
=\int_{\Omega} F\varphi\, dx,
\end{equation*}
for all $\varphi \in W^{s,p}_0(\Omega)$. Then, for every convex $C^1$ function $\Phi:{\mathbb R}\to {\mathbb R}$, we have
\begin{equation*}
\int_{\mathbb{R}^{2N}} \frac{|\Phi(u)(x) - \Phi(u)(y)|^{p-2}\, (\Phi(u)(x) - \Phi(u)(y))\, (\varphi(x) - \varphi(y))}{|x - y|^{N+s\,p}}\, dx\, dy\leq
\int_{\Omega} F|\Phi'(u)|^{p-2}\Phi'(u)\varphi\, dx,
\end{equation*}
for every nonnegative function $\varphi \in W^{s,p}_0(\Omega)$.
\end{proposition}

\begin{lemma}
\label{lem:apiori}
Let $\{u_n\}_{n\in {\mathbb N}}\subset W^{s,p}_0(\Omega)\cap L^\infty(\Omega)$ be the sequence of
solution to \eqref{prob-approx} provided by Proposition \ref{pro:solapprox}. Let
$\gamma>1$ and $f\in L^{1}(\Omega)$. Then $\{u_n^{(\gamma+p-1)/p}\}_{n\in {\mathbb N}}$ is bounded in $W^{s,p}_0(\Omega)$.
\end{lemma}
\begin{proof}	
We can apply Proposition~\ref{convex} to each $u_n\geq 0$ by choosing
$$
F(x):=\frac{f_n(x)}{(u_n+1/n)^\gamma}\in L^\infty(\Omega),\qquad \Phi(s):=s^{(\gamma+p-1)/p},\,\,\, s\geq 0,
$$
by noticing that $\Phi$ is $C^1$ and convex on ${\mathbb R}^+$ since $\gamma>1$. Then
	\begin{align*}
&	\int_{\mathbb{R}^{2N}} \frac{|\Phi(u_n)(x) - \Phi(u_n)(y)|^{p-2}\, (\Phi(u_n)(x) - \Phi(u_n)(y))\, (\varphi(x) - \varphi(y))}{|x - y|^{N+s\,p}}\, dx\, dy \\
	&\leq\int_{\Omega} \frac{f_n(x)}{(u_n+1/n)^\gamma}|\Phi'(u_n)|^{p-2}\Phi'(u_n)\varphi\, dx,
	\end{align*}
	for every $n$ and all nonnegative function $\varphi \in W^{s,p}_0(\Omega)$. By choosing $\varphi:=\Phi(u_n)$ as test function
	(which belongs to $W^{s,p}_0(\Omega),$ since $\Phi$ is Lipschitz on bounded intervals), we infer
		\begin{equation}
		\label{bound}
			\int_{\mathbb{R}^{2N}} \frac{|\Phi(u_n)(x) - \Phi(u_n)(y)|^{p}}{|x - y|^{N+s\,p}}\, dx\, dy
		\leq\int_{\Omega} \frac{f_n(x)}{(u_n+1/n)^\gamma}|\Phi'(u_n)|^{p-1}\Phi(u_n)\, dx.
		\end{equation}
		Note that
		$$
		|\Phi'(u_n)|^{p-1}\Phi(u_n)\leq C u_n^\gamma,\quad\,\,\,
		\text{for any $n\in {\mathbb N}$ and some $C=C(\gamma,p)>0$.}
		$$
		In turn,
		$$
		\int_{\Omega} \frac{f_n(x)}{(u_n+1/n)^\gamma}|\Phi'(u_n)|^{p-1}\Phi(u_n)\, dx\leq C\int_{\Omega} |f_n|\, dx\leq C\int_{\Omega} |f|\, dx,
		$$
		since $f_n\leq f$. From inequality~\eqref{bound} it follows that
		$\{\Phi(u_n)\}_{n\in {\mathbb N}}$ is bounded in $W^{s,p}_0(\Omega)$.
\end{proof}

\begin{lemma}\label{lem:dino}
	Let $q>1$ and $\e>0$. In  the plane $\mathbb{R}^2$ with the notation $p=(x,y)$, let us set
	\[
	S_\e^x\,:=\{x\geq \e \}\cap\{y\geq 0\}\qquad  S_\e^y\,:=\{y\geq \e \}\cap\{x\geq 0\}\,.
	\]
	Then, we have that
	\begin{equation}\label{fgjsfgsjf}
	|x^q-y^q|\geq \e^{q-1}|x-y|\qquad\text{in}\quad S_\e^x\cup S_\e^y\,.
	\end{equation}
\end{lemma}\label{vicesindaco}
\begin{proof}
	With no loss of generality, we may assume that $x\geq y$.
	Let us first note that
	\begin{equation}\nonumber
	x^q-y^q=q\lambda^{q-1}(x-y),\qquad\text{for some } \lambda\in (y,x)\,.
	\end{equation}
	Whence \eqref{fgjsfgsjf} holds true, since $q>1$, if $(x,y)\in  S_\e^x\cap S_\e^y$, namely if $y\geq \e$ in the case that we are considering.
	Then, let us deal with the case $0\leq y<\e\leq x$. Since $t\mapsto t^q$ is (strictly) convex for $q>1$, then we have
	\begin{equation}\nonumber
	\frac{x^q-y^q}{x-y}\geq \frac{x^q}{x}=x^{q-1}\geq \e^{q-1}\,.
	\end{equation}
	Thus, inequality  \eqref{fgjsfgsjf} is proved.
\end{proof}

\noindent
Next we turn to the existence result for $\gamma>1$.

\begin{theorem}[Existence for $\gamma>1$]
	\label{thm:exgamma>1}
	Let $f\geq0$, $f\in L^1(\Omega)$ and $\gamma>1$.	Then problem \eqref{prob} admits a weak a solution $u \in W^{s,p}_{{\rm loc}}(\Omega)$
	with $u^{(\gamma+p-1)/p}\in W^{s,p}_0(\Omega)$.
\end{theorem}
\begin{proof}
	In light of Lemma~\ref{lem:apiori}, the sequence $\{u_n\}_{n\in {\mathbb N}}$ of
	solution to \eqref{prob-approx} of Proposition \ref{pro:solapprox} satisfies
	\begin{equation}\label{eq:gerry1}
	\sup_{n\in {\mathbb N}}\left[u_n^{(\gamma+p-1)/p}\right]_{W^{s,p}(\mathbb{R}^N)} \leq C.
	\end{equation}
Since $\{u_n\}_{n\in {\mathbb N}}$ is increasing it admits pointwise limit $u$ as $n\to\infty$. In particular, by Fatou's lemma
\begin{equation}\label{eq:gerry2}
\left[u^{(\gamma+p-1)/p}\right]_{W^{s,p}(\mathbb{R}^N)}\leq \liminf_n\left[u_n^{(\gamma+p-1)/p}\right]_{W^{s,p}(\mathbb{R}^N)} \leq C.
\end{equation}
Then $u^{(\gamma+p-1)/p}\in W^{s,p}_0(\Omega)$ and, in particular $u\in L^p(\Omega)$ via the Sobolev embedding.
Notice also that, by virtue of Lemma~\ref{lem:monotonicity},
for all $K\Subset\Omega$ there exists $\sigma_K>0$ such that $u(x)\geq \sigma_K>0$ for a.e.\ $x\in K$.
Therefore, in light of Lemma \ref{lem:dino}, we have
\begin{equation*}
\frac{|u(x) - u(y)|^{p}}{|x - y|^{N+s\,p}}\leq \sigma^{1-\gamma}_K\frac{|u^{\frac{\gamma+p-1}{p}}(x)-u^{\frac{\gamma+p-1}{p}}(y)|^p}{|x - y|^{N+s\,p}},
\qquad x,y\in K,\quad K\Subset{\mathbb R} ^{N}.
\end{equation*}
This yields
$$u \in W^{s,p}_{{\rm loc}}(\Omega).$$ We have, for any $n\in {\mathbb N}$
\begin{equation}\label{eq:equazPn-2}
\int_{\mathbb{R}^{2N}} \frac{|u_n(x) - u_n(y)|^{p-2}\, (u_n(x) - u_n(y))\, (\varphi(x) - \varphi(y))}{|x - y|^{N+s\,p}}\, dx\, dy
=\int_{\Omega} \frac{f_n(x)}{(u_n+1/n)^\gamma} \,\varphi\, dx,
\end{equation}
for all $\varphi \in C^\infty_c(\Omega)$. In order to pass to the limit in \eqref{eq:equazPn-2}, we observe the following.
By the elementary inequality $||\xi|^{p-2}\xi-|\xi'|^{p-2}\xi'|\leq C(|\xi|+|\xi'|)^{p-2}|\xi-\xi'|$ for $\xi,\xi'\in{\mathbb R}$ with $|\xi|+|\xi'|>0,$ we get
\begin{eqnarray}\label{eq:un-u}
&&\Bigg|\int_{\mathbb{R}^{2N}} \frac{|u_n(x) - u_n(y)|^{p-2}\, (u_n(x) - u_n(y))\, (\varphi(x) - \varphi(y))}{|x - y|^{N+s\,p}}\, dx\, dy-\\\nonumber &&\int_{\mathbb{R}^{2N}} \frac{|u(x) - u(y)|^{p-2}\, (u(x) - u(y))\, (\varphi(x) - \varphi(y))}{|x - y|^{N+s\,p}}\, dx\, dy\Bigg|\\\nonumber
&&\leq \int_{\mathbb{R}^{2N}}  \frac{(|u_n(x) - u_n(y)|+|u(x) - u(y)|)^{p-2}|\bar u_n(x)-\bar u_n(y)||\varphi(x) - \varphi(y)|}{|x - y|^{N+s\,p}}\, dx\, dy,
\end{eqnarray}
where $\bar u_n(x):= u_n(x)-u(x)$.
Let us fix $\varepsilon >0$.

We {\em claim} that there exist a compact $\mathcal K \subset \mathbb{R}^{2N}$ such that
\begin{equation}\label{eq:un-u1}
 \int_{\mathbb{R}^{2N}\setminus \mathcal K}  \frac{(|u_n(x) - u_n(y)|+|u(x) - u(y)|)^{p-2}|\bar u_n(x)-\bar u_n(y)||\varphi(x) - \varphi(y)|}{|x - y|^{N+s\,p}}\, dx\, dy\leq \frac{\varepsilon}{2},
\end{equation}
for all $n\in \mathbb N$. Let us set
\[
\mathcal S_\varphi\,:={\rm supp} \,\varphi \qquad \mathcal Q_\varphi\,:=\,\mathbb{R}^{2N}\setminus \big(\mathcal S_\varphi^c\times\mathcal S_\varphi^c\big)\,.
\]
By triangular  and H\"older inequalities we get
\begin{eqnarray}\label{eq:un-u2}\\\nonumber
&&\int_{\mathbb{R}^{2N}\setminus \mathcal K}  \frac{(|u_n(x) - u_n(y)|+|u(x) - u(y)|)^{p-2}|\bar u_n(x)-\bar u_n(y)||\varphi(x) - \varphi(y)|}{|x - y|^{N+s\,p}}\, dx\, dy\\\nonumber
&&= \int_{\mathcal Q_\varphi\setminus \mathcal K}  \frac{(|u_n(x) - u_n(y)|+|u(x) - u(y)|)^{p-1}|\varphi(x) - \varphi(y)|}{|x - y|^{N+s\,p}}\, dx\, dy\\\nonumber
&&\leq \left(\int_{\mathcal Q_\varphi\setminus \mathcal K}  \frac{(|u_n(x) - u_n(y)|+|u(x) - u(y)|)^{p}}{|x - y|^{N+s\,p}}\, dx\, dy\right)^{\frac{p-1}{p}}\left(\int_{\mathcal Q_\varphi\setminus \mathcal K} \frac{|\varphi(x) - \varphi(y)|^{p}}{|x - y|^{N+s\,p}}\, dx\, dy \right)^{\frac 1p}\\\nonumber
&&=\left(\int_{\mathcal Q_\varphi\setminus \mathcal K} \frac{(|u_n(x) - u_n(y)|+|u(x) - u(y)|)^{p}}{|x - y|^{N+s\,p}}\, dx\, dy\right)^{\frac{p-1}{p}}\left(\int_{\mathbb{R}^{2N}\setminus \mathcal K} \frac{|\varphi(x) - \varphi(y)|^{p}}{|x - y|^{N+s\,p}}\, dx\, dy \right)^{\frac 1p}.
\end{eqnarray}
By Lemma \ref{lem:monotonicity}, there exists $\sigma_{\mathcal S_\varphi}>0$ independent of $n$ such that $u(x)\geq \sigma_{\mathcal S_\varphi}$ for a.e. $x
\in \mathcal S_\varphi$. Moreover, by using Lemma \ref{lem:dino}  with $q={(\gamma+p-1)}/{p}$, we have
\begin{align}\label{eq:lagrangedino}
&\frac{(|u_n(x) - u_n(y)|+|u(x) - u(y)|)^{p}}{|x - y|^{N+s\,p}}  \nonumber \\
&\leq C(p)\frac{|u_n(x) - u_n(y)|^p+|u(x) - u(y)|^{p}}{|x - y|^{N+s\,p}}\\\nonumber
&\leq C(p)\sigma_{\mathcal S_\varphi}^{1-
\gamma}\frac{|u_n^{\frac{\gamma+p-1}{p}}(x)-u_n^{\frac{\gamma+p-1}{p}}(y)|^p+|u^{\frac{\gamma+p-1}{p}}(x)-u^{\frac{\gamma+p-1}{p}}(y)|^p}{|x - y|^{N+s\,p}},\quad\,\,\text{a.e.\ $(x,y)\in \mathcal Q_\varphi$.}
\end{align}
Then from  \eqref{eq:un-u2} and \eqref{eq:lagrangedino} we infer that
\begin{eqnarray}\nonumber\\\nonumber
&&\int_{\mathbb{R}^{2N}\setminus \mathcal K}  \frac{(|u_n(x) - u_n(y)|+|u(x) - u(y)|)^{p-2}|\bar u_n(x)-\bar u_n(y)||\varphi(x) - \varphi(y)|}{|x - y|^{N+s\,p}}\, dx\, dy\\\nonumber
&&\leq C\left(\int_{\mathbb{R}^{2N}} \frac{|u_n^{\frac{\gamma+p-1}{p}}(x)-u_n^{\frac{\gamma+p-1}{p}}(y)|^p}{|x - y|^{N+s\,p}}\, dx\, dy\right)^{\frac{p-1}{p}}\left(\int_{\mathbb{R}^{2N}\setminus \mathcal K} \frac{(|\varphi(x) - \varphi(y)|)^{p}}{|x - y|^{N+s\,p}}\, dx\, dy \right)^{\frac 1p}\\\nonumber
&&+ C\left(\int_{\mathbb{R}^{2N}} \frac{|u^{\frac{\gamma+p-1}{p}}(x)-u^{\frac{\gamma+p-1}{p}}(y)|^p}{|x - y|^{N+s\,p}}\, dx\, dy\right)^{\frac{p-1}{p}}\left(\int_{\mathbb{R}^{2N}\setminus \mathcal K} \frac{(|\varphi(x) - \varphi(y)|)^{p}}{|x - y|^{N+s\,p}}\, dx\, dy \right)^{\frac 1p}
\\\nonumber &&\leq C \left(\int_{\mathbb{R}^{2N}\setminus \mathcal K} \frac{|\varphi(x) - \varphi(y)|^{p}}{|x - y|^{N+s\,p}}\, dx\, dy \right)^{\frac 1p},
\end{eqnarray}
with $C=C(p,\gamma,\sigma_{\mathcal S_\varphi})$ and where we  used \eqref{eq:gerry1} and \eqref{eq:gerry2}. Then, since $\varphi \in C^\infty_c(\Omega)$, there exists  $\mathcal K=\mathcal K(\varepsilon)$ such that \eqref{eq:un-u1} holds, proving the {\em claim}.

On the other hand,  consider now an arbitrary measurable subset $E\subset {\mathcal K}$. Arguing as in \eqref{eq:un-u2} and \eqref{eq:lagrangedino},
we reach the inequality
\begin{align}\nonumber
&\int_{E}  \frac{(|u_n(x) - u_n(y)|+|u(x) - u(y)|)^{p-2}|\bar u_n(x)-\bar u_n(y)||\varphi(x) - \varphi(y)|}{|x - y|^{N+s\,p}}\, dx\, dy\\\nonumber
& \leq C \left(\int_{E}\frac{|\varphi(x) - \varphi(y)|^{p}}{|x - y|^{N+s\,p}}\, dx\, dy \right)^{\frac 1p},
\end{align}
that is
\begin{equation}\nonumber
\int_{E}  \frac{(|u_n(x) - u_n(y)|+|u(x) - u(y)|)^{p-2}|\bar u_n(x)-\bar u_n(y)||\varphi(x) - \varphi(y)|}{|x - y|^{N+s\,p}}\, dx\, dy\rightarrow 0,
\end{equation}
uniformly on $n$, if the  Lebesgue measure of $E$ goes to zero. Moreover
\begin{equation*}
\frac{(|u_n(x) - u_n(y)|+|u(x) - u(y)|)^{p-2}|\bar u_n(x)-\bar u_n(y)||\varphi(x) - \varphi(y)|}{|x - y|^{N+s\,p}}\rightarrow 0\quad \text{a.e.\ in ${\mathbb R}^{2N}$}.
\end{equation*}
Vitali's Theorem now implies that, given $\varepsilon >0$,  there exists $\bar n >0$ such that, if $n\geq \bar n$, it follows
\begin{equation}\label{eq:un-u3}
\int_{\mathcal K}  \frac{(|u_n(x) - u_n(y)|+|u(x) - u(y)|)^{p-2}|\bar u_n(x)-\bar u_n(y)||\varphi(x) - \varphi(y)|}{|x - y|^{N+s\,p}}\, dx\, dy\leq \frac{\varepsilon}{2}.
\end{equation}
From \eqref{eq:un-u}, using \eqref{eq:un-u1} and \eqref{eq:un-u3},
we are able to pass to the limit in the left-hand side of \eqref{prob-approx}, that is
\begin{eqnarray}\nonumber
&&\lim_{n \rightarrow +\infty}\int_{\mathbb{R}^{2N}} \frac{|u_n(x) - u_n(y)|^{p-2}\, (u_n(x) - u_n(y))\, (\varphi(x) - \varphi(y))}{|x - y|^{N+s\,p}}\, dx\, dy
=
\\\nonumber
&&\int_{\mathbb{R}^{2N}} \frac{|u(x) - u(y)|^{p-2}\, (u(x) - u(y))\, (\varphi(x) - \varphi(y))}{|x - y|^{N+s\,p}}\, dx\, dy,
\end{eqnarray}
for all $\varphi \in C^\infty_c(\Omega)$.
Finally, arguing for the right-hand side as in the proof of 	Theorem~\ref{thm:exgamma<=1}, we pass to the limit in \eqref{eq:equazPn},  concluding  that
\begin{equation*}
\int_{\mathbb{R}^{2N}} \frac{|u(x) - u(y)|^{p-2}\, (u(x) - u(y))\, (\varphi(x) - \varphi(y))}{|x - y|^{N+s\,p}}\, dx\, dy
=\int_{\Omega} \frac{f(x)}{u^\gamma}\,\varphi\, dx,
\end{equation*}
for all $\varphi \in C^\infty_c(\Omega)$, namely $u$ is a solution to \eqref{prob}.
\end{proof}

\begin{remark}
	In the previous  results, if furthermore
	$$
	f\in L^{\frac{Np}{N(p-1)+sp}}(\Omega),
	$$
	then \eqref{v-f} is satisfied for all $\varphi \in W^{s,p}_0(\Omega)$ such that ${\rm supp}(\varphi) \Subset  \Omega$.
	\end{remark}

\begin{remark}
	With reference to the proof of Theorem~\ref{thm:exgamma>1}, we observe that
	$$
	\text{the sequence $\{u_n\}_{n\in {\mathbb N}}$ is bounded in $W^{\frac{sp}{\gamma+p-1}, \gamma+p-1}_0(\Omega)$.}
	$$
	In fact, since $(\gamma+p-1)/p>1$, by the H\"olderianity of the map $t\mapsto t^{p/(\gamma+p-1)}$, we get
	\begin{align*}
	\frac{|u_n(x)-u_n(y)|^{\gamma+p-1}}{|x-y|^{N+sp}}
	\leq \frac{|u_n^{(\gamma+p-1)/p}(x)-u_n^{(\gamma+p-1)/p}(y)|^p}{|x-y|^{N+sp}},\quad x,y \in {\mathbb R}^N.
	\end{align*}
	Therefore, the weak solution $u$ of Theorem~\ref{thm:exgamma>1}
	also belongs to $u\in W^{\frac{sp}{\gamma+p-1}, \gamma+p-1}_0(\Omega).$
\end{remark}

\section{Proof of  the uniqueness  results}
\label{kduhfkuer}

\noindent
Let us start defining the real valued function $g_{k}$ by
\begin{equation}\nonumber
g_{k}(s)\,:=\,
\begin{cases}
\min\{s^{-\beta}\,,\,k\}\qquad\text{if $s>0$},\\
k\qquad\qquad\quad\qquad\text{if $s\leq 0$}.
\end{cases}
\end{equation}
Then we consider the real valued function $\Phi_{k}$ defined to be the primitive of $g_{k}$ that is equal to zero for $s=1$.
Let us consequently consider the functional $J_k\,:W^{s,p}_0(\Omega)\rightarrow [-\infty\,,\,+\infty]$ defined by
\begin{equation}\nonumber
J_k(\varphi)\,:=\,\frac{1}{p}\int_{\mathbb{R}^{2N}}\frac{|\varphi(x)-\varphi(y)|^p}{|x-y|^{N+ps}}\,dx\,dy-\int_{\mathbb{R}^{N}} f(x)\Phi_{k}(\varphi)\,dx\qquad\varphi\in W^{s,p}_0(\Omega)\,.
\end{equation}
Let us now recall that, given $z\in W^{s,p}_{{\rm loc}}(\Omega)\cap L^{p-1}(\Omega)$ with   $z\geq 0$, we say that $z$ is    a weak supersolution (subsolution) to
\eqref{prob}, if
\begin{equation*}
\int_{\mathbb{R}^{2N}} \frac{|z(x) - z(y)|^{p-2}\, (z(x) - z(y))\, (\varphi(x) - \varphi(y))}{|x - y|^{N+s\,p}}\, dx\, dy
\underset{(\leq)}{\geq}\int_{\Omega} \frac{f(x)}{z^\gamma}\,\varphi\, dx\qquad \forall \varphi\in C^\infty_c(\Omega)\,,\, \varphi\geq 0\,.
\end{equation*}
For a fixed supersolution $v$, we consider  $w$  defined as the minimum of $J_k$ on the convex set
\[
\mathcal K\,:=\,\{\varphi\in W^{s,p}_0(\Omega)\,:\, 0\leq \varphi\leq v\,\,\text{a.e. in}\,\,\Omega\}\,.
\]
By direct computation, we deduce that

\begin{equation}\label{eqminxx}
\begin{split}
&\int_{\mathbb{R}^{2N}} \frac{|w(x) - w(y)|^{p-2}\, (w(x) - w(y))\, (\psi(x)-w(x)-(\psi(y)-w(y)))}{|x - y|^{N+s\,p}}\, dx\, dy\\
&\geq \int_\Omega\,f(x) \Phi_{k}'(w)(\psi-w)\quad\,\,\, \text{for}\,\,
 \psi\in w+\left(W^{s,p}_0(\Omega)\cap L^\infty_c(\Omega)\right)\,\,\text{and}\,\,0\leq \psi\leq v\,.
 \end{split}
\end{equation}
With such notation, we have the following

\begin{lemma}\label{lemmause}
For all $\psi\in C^\infty_c(\Omega)$ with $\psi\geq 0$ we have
\begin{equation}\label{ksjfskgfdfjigf}
\int_{\mathbb{R}^{2N}} \frac{|w(x) - w(y)|^{p-2}\, (w(x) - w(y))\, (\psi(x)-\psi(y))}{|x - y|^{N+s\,p}}\, dx\, dy\\
\geq \int_\Omega\,f(x) \Phi_{k}'(w)\psi dx.
\end{equation}
\end{lemma}
\begin{proof}
Let us consider a real valued function
 $g\in C^\infty_c(\mathbb{R})$ with $0\leq g(t) \leq 1$, $g (t)=1$
for $t\in[-1,1]$ and $g (t)=0$
for $t\in(-\infty,-2]\cup[2,\infty)$. Then, for any nonnegative $\varphi\in C^\infty_c(\Omega)$,
we set $\varphi_h\,:=\,g(\frac{w}{h})\,\varphi$ and $\varphi_{h,t}\,:=\,\min\{w+t\varphi_h\,,\,v\}$
with $h\geq 1$ and $t>0$.
We have that $\varphi_{h,t}\in w+\left(W^{s,p}_0(\Omega)\cap L^\infty_c(\Omega)\right)$ and
$0\leq \varphi_{h,t}\leq v$, so that, by \eqref{eqminxx}, we deduce that
\begin{equation*}
\begin{split}
&\int_{\mathbb{R}^{2N}} \frac{|w(x) - w(y)|^{p-2}\, (w(x) - w(y))\, (\varphi_{h,t}(x)-w(x)-(\varphi_{h,t}(y)-w(y)))}{|x - y|^{N+s\,p}}\, dx\, dy\\
&\geq \int_\Omega\,f(x) \Phi_{k}'(w)(\varphi_{h,t}-w)\,dx.
 \end{split}
\end{equation*}
By standard manipulations and by \eqref{eqminxx}, we deduce that
\begin{equation*}
\begin{split}
& {\mathbb I}_1:=\,c \int_{\mathbb{R}^{2N}} \frac{(|\varphi_{h,t}(x) - \varphi_{h,t}(y)|+|w(x) - w(y)|)^{p-2}\, (\varphi_{h,t}(x)-w(x) - (\varphi_{h,t}(y)-w(y)))^2\, }{|x - y|^{N+s\,p}}\, dx\, dy\\
&\leq \int_{\mathbb{R}^{2N}} \frac{|\varphi_{h,t}(x) - \varphi_{h,t}(y)|^{p-2}\, (\varphi_{h,t}(x) - \varphi_{h,t}(y))\, (\varphi_{h,t}(x)-w(x)-(\varphi_{h,t}(y)-w(y)))}{|x - y|^{N+s\,p}}\, dx\, dy\\
&-\int_{\mathbb{R}^{2N}} \frac{|w(x) - w(y)|^{p-2}\, (w(x) - w(y))\, (\varphi_{h,t}(x)-w(x)-(\varphi_{h,t}(y)-w(y)))}{|x - y|^{N+s\,p}}\, dx\, dy\\
&\leq \int_{\mathbb{R}^{2N}} \frac{|\varphi_{h,t}(x) - \varphi_{h,t}(y)|^{p-2}\, (\varphi_{h,t}(x) - \varphi_{h,t}(y))\, (\varphi_{h,t}(x)-w(x)-(\varphi_{h,t}(y)-w(y)))}{|x - y|^{N+s\,p}}\, dx\, dy\\
&-\int_\Omega\,f(x) \Phi_{k}'(w)(\varphi_{h,t}-w)
 \end{split}
\end{equation*}
We rewrite this as
\begin{equation}\label{cicciofriccio}
\begin{split}
&\,\,{\mathbb I}_1-\int_\Omega\,f(x)(\Phi_{k}'(\varphi_{h,t})-\Phi_{k}'(w))(\varphi_{h,t}-w)\,dx\\
&\leq \int_{\mathbb{R}^{2N}} \frac{|\varphi_{h,t}(x) - \varphi_{h,t}(y)|^{p-2}\, (\varphi_{h,t}(x) - \varphi_{h,t}(y))\, (\varphi_{h,t}(x)-w(x)-(\varphi_{h,t}(y)-w(y)))}{|x - y|^{N+s\,p}}\, dx\, dy\\
&-\int_\Omega\,f(x)\cdot \Phi_{k}'(\varphi_{h,t})(\varphi_{h,t}-w)\\
&=\int_{\mathbb{R}^{2N}}\mathcal G(x,y)\, dx dy
-\int_\Omega\,f(x)\cdot \Phi_{k}'(\varphi_{h,t})(\varphi_{h,t}-w-t\varphi_h)\\
&+t\int_{\mathbb{R}^{2N}} \frac{|\varphi_{h,t}(x) - \varphi_{h,t}(y)|^{p-2}\, (\varphi_{h,t}(x) - \varphi_{h,t}(y))\, (\varphi_h(x)-\varphi_h(y))}{|x - y|^{N+s\,p}}\, dx\, dy\\
&-t\int_\Omega\,f(x) \Phi_{k}'(\varphi_{h,t})\varphi_h\\
 \end{split}
\end{equation}
where, if $I_p(t):=|t|^{p-2}t$, we have set
\begin{equation*}
\mathcal G(x,y)\,:=
\frac{I_p(\varphi_{h,t}(x) - \varphi_{h,t}(y))\, (\varphi_{h,t}(x)-w(x)-t\varphi_h(x)-(\varphi_{h,t}(y)-w(y)-t\varphi_h(y)))}{|x - y|^{N+s\,p}}\,.
\end{equation*}
For future use, let us also set
\begin{equation*}
 \mathcal G_v(x,y)\,:=
\frac{I_p(v(x) - v(y))\, (\varphi_{h,t}(x)-w(x)-t\varphi_h(x)-(\varphi_{h,t}(y)-w(y)-t\varphi_h(y)))}{|x - y|^{N+s\,p}}\,.
\end{equation*}
Now we set $S_v:=\{\varphi_{h,t}=v\}$ and note that, actually, $S_v=\{v\leq w+t\varphi_h\}$.
We use the decomposition
$$
\mathbb{R}^{2N}=\left(S_v\cup S_v^c\right)\times \left(S_v\cup S_v^c\right).
$$
Taking into account that
$$\mathcal G(\cdot, \cdot)=0 \,\,\text{in}\,\, S_v^c\times S_v^c,$$
we deduce that
\begin{equation*}
\begin{split}
&\int_{\mathbb{R}^{2N}}\mathcal G(x,y)\, dx\, dy\\
&=\int_{S_v}\int_{S_v}\mathcal G(x,y)\, dx\, dy+\int_{S_v^c}\int_{S_v}\mathcal G(x,y)\, dx\, dy
+\int_{S_v}\int_{S_v^c}\mathcal G(x,y)\, dx\, dy\\
&\leq \int_{S_v}\int_{S_v}\mathcal G_v(x,y)\, dx\, dy+\int_{S_v^c}\int_{S_v}\mathcal G_v(x,y)\, dx\, dy
+\int_{S_v}\int_{S_v^c}\mathcal G_v(x,y)\, dx\, dy\\
&= \int_{\mathbb{R}^{2N}} \mathcal G_v(x,y)\, dx\, dy,
 \end{split}
\end{equation*}
where the  inequality follows since $\mathcal G\leq \mathcal G_v$.  In particular, to see this,
write down explicitly the expression of $\mathcal G$ and exploit the monotonicity of the real valued function $|t-t_0|^{p-2}(t-t_0)$ together with the definition of $\varphi_{h,t}$. Thence we go back to \eqref{cicciofriccio} and get that
\begin{equation}\nonumber
\begin{split}
&\,\,{\mathbb I}_1-\int_\Omega\,f(x)\cdot(\Phi_{k}'(\varphi_{h,t})-\Phi_{k}'(w))(\varphi_{h,t}-w)\,dx\\
&\leq \int_{\mathbb{R}^{2N}}\mathcal G_v(x,y)\, dx\, dy-\int_\Omega\,f(x)\cdot \Phi_{k}'(\varphi_{h,t})(\varphi_{h,t}-w-t\varphi_h)\\
&+t\int_{\mathbb{R}^{2N}} \frac{|\varphi_{h,t}(x) - \varphi_{h,t}(y)|^{p-2}\, (\varphi_{h,t}(x) - \varphi_{h,t}(y))\, (\varphi_h(x)-\varphi_h(y))}{|x - y|^{N+s\,p}}\, dx\, dy\\
&-t\int_\Omega\,f(x)\cdot \Phi_{k}'(\varphi_{h,t})\varphi_h\,.\\
 \end{split}
\end{equation}
By the definition of $\Phi_{k}$, it follows that $v$ is  a supersolution to the equation
$(-\Delta)^s_p z=\Phi_{k}'(z)$ too. Therefore, recalling  that $\varphi_{h,t}-w- t\varphi_h\leqslant 0$, we deduce that
\begin{equation}\nonumber
\begin{split}
&\,\,{\mathbb I}_1-\int_\Omega\,f(x)\cdot(\Phi_{k}'(\varphi_{h,t})-\Phi_{k}'(w))(\varphi_{h,t}-w)\,dx\\
&\leq t\int_{\mathbb{R}^{2N}} \frac{|\varphi_{h,t}(x) - \varphi_{h,t}(y)|^{p-2}\, (\varphi_{h,t}(x) - \varphi_{h,t}(y))\, (\varphi_h(x)-\varphi_h(y))}{|x - y|^{N+s\,p}}\, dx\, dy-t\int_\Omega\,f(x)\cdot \Phi_{k}'(\varphi_{h,t})\varphi_h\,.\\
 \end{split}
\end{equation}
Exploiting again the fact  that $\varphi_{h,t}-w\leqslant t\varphi_h$,  we deduce that
\begin{equation}\nonumber
\begin{split}
&\int_{\mathbb{R}^{2N}} \frac{|\varphi_{h,t}(x) - \varphi_{h,t}(y)|^{p-2}\, (\varphi_{h,t}(x) - \varphi_{h,t}(y))\, (\varphi_h(x)-\varphi_h(y))}{|x - y|^{N+s\,p}}\, dx\, dy-\int_\Omega\,f(x)\cdot \Phi_{k}'(\varphi_{h,t})\varphi_h\\
&\geq -\int_\Omega
 f(x)\cdot|\Phi_{k}'(\varphi_{h,t})-\Phi_{k}'(w)||\varphi_{h}|\,dx\,.\\
 \end{split}
\end{equation}
We can now pass to the limit for $t\rightarrow 0$ exploiting also the Lebesgue Theorem obtaining
\begin{equation}\nonumber
\begin{split}
&\int_{\mathbb{R}^{2N}} \frac{|w(x) - w(y)|^{p-2}\, (w(x) - w(y))\, (\varphi_h(x)-\varphi_h(y))}{|x - y|^{N+s\,p}}\, dx\, dy-\int_\Omega\,f(x)\cdot \Phi_{k}'(w)\varphi_h\geq 0.
 \end{split}
\end{equation}
The claim, namely the proof of \eqref{ksjfskgfdfjigf}, follows letting $h\to\infty$.
\end{proof}

\noindent
Now we are in position to prove our \emph{weak comparison principle}, namely we have the following:
\begin{theorem}\label{comparison}
Let $\gamma>0$ and let $f\in L^1(\Omega)$ be  non-negative. Let $u$ be a subsolution to  \eqref{prob} such that $u\leq 0$ on $\partial\Omega$ and let $v$ be a supersolution to
\eqref{prob}. Then, $u\leq v$ a.e. in $\Omega$.
\end{theorem}
\begin{proof}
For $\e>0$ and $w$ as in Lemma \ref{lemmause}, it follows that
\[
(u-w-\e)^+\in W^{s,p}_0(\Omega)\,.
\]
This can be easily deduced by the fact that  $w\in W^{s,p}_0(\Omega)$ and $w\geqslant 0$ a.e. in $\Omega$, so that the support of $(u-w-\e)^+$ is contained in the support of $(u-\e)^+$.
Therefore, by \eqref{ksjfskgfdfjigf} and by standard density arguments, it follows
\begin{equation}\label{eq111}
\begin{split}
&\int_{\mathbb{R}^{2N}} \frac{|w(x) - w(y)|^{p-2}\, (w(x) - w(y))\, (T_\tau\left((u-w-\e)^+\right)(x)-T_\tau\left((u-w-\e)^+\right)(y))}{|x - y|^{N+s\,p}}\, dx\, dy\\
&\geq \int_\Omega\,f(x)\cdot \Phi_{k}'(w)T_\tau\left((u-w-\e)^+\right)
 \end{split}
\end{equation}
for $T_\tau (s)\,:=\, \min\{s,\tau\}$ for $s\geq 0$ and $T_\tau (-s)\,:=\,-T_\tau (s)$ for $s< 0$.
Let now $\varphi_n\in C^\infty_c(\Omega)$ such that $\varphi_n\rightarrow (u-w-\e)^+$ in $W^{s,p}_0(\Omega)$ and set
\[
\tilde\varphi_{\tau,n}\,:=\,T_\tau (\min\{(u-w-\e)^+,\varphi_n^+\})\,.
\]
It follows that $\tilde\varphi_{\tau,n}\in W^{s,p}_0(\Omega)\cap L^\infty_c(\Omega)$ so that, by a density argument 
\begin{equation}\nonumber
\int_{\mathbb{R}^{2N}} \frac{|u(x) - u(y)|^{p-2}\, (u(x) - u(y))\, (\tilde\varphi_{\tau,n}(x)-\tilde\varphi_{\tau,n}(y))}{|x - y|^{N+s\,p}}\, dx\, dy
 \leq \int_\Omega\,\frac{f(x)}{u^\gamma}\,\tilde\varphi_{\tau,n}\,dx\,.
\end{equation}
Passing to the limit as $n$ tends to infinity, it is easy to deduce that
\begin{equation}\label{eq222}
\begin{split}
 &\int_{\mathbb{R}^{2N}} \frac{|u(x) - u(y)|^{p-2}\, (u(x) - u(y))\, (T_\tau\left((u-w-\e)^+(x)\right)-T_\tau\left((u-w-\e)^+(y)\right))}{|x - y|^{N+s\,p}}\, dx\, dy\\
 &\leq \int_\Omega\frac{f(x)}{u^\gamma}T_\tau\left((u-w-\e)^+\right)\,dx\,.
 \end{split}
\end{equation}
It is convenient now to set
\[
g(t)\,:=\,T_\tau((t-\e)^+)=\min\{\tau\,,\,\max\{t-\e\,,\,0\}\}\,.
\]
With such a notation we have
\begin{equation}\nonumber
\begin{split}
&|u(x) - u(y)|^{p-2}\, (u(x) - u(y))\, (T_\tau\left((u-w-\e)^+(x)\right)-T_\tau\left((u-w-\e)^+(y)\right))=\\
&=|u(x) - u(y)|^{p-2}\, (u(x) - u(y))(u(x)-w(x)-(u(y)-w(y))\, H(x,y)
 \end{split}
\end{equation}
with
\[
H(x,y):=\frac{g(u(x)-w(x))-g(u(y)-w(y))}{(u(x)-w(x)-(u(y)-w(y))}\,.
\]
In the same way we deduce that
\begin{equation}\nonumber
\begin{split}
&|w(x) - w(y)|^{p-2}\, (w(x) - w(y))\, (T_\tau\left((u-w-\e)^+(x)\right)-T_\tau\left((u-w-\e)^+(y)\right))=\\
&=|w(x) - w(y)|^{p-2}\, (w(x) - w(y))(u(x)-w(x)-(u(y)-w(y))\, H(x,y)\,.
 \end{split}
\end{equation}
Then, we subtract \eqref{eq111} to \eqref{eq222}, choosing $\e>0$ such that $\e^{-\beta}<k$ and we deduce that
\begin{equation}\nonumber
\begin{split}
&c \int_{\mathbb{R}^{2N}} \frac{(|u(x) - u(y)|+|w(x) - w(y)|)^{p-2}\, (u(x)-w(x) - (u(y)-w(y)))^2\, }{|x - y|^{N+s\,p}}H(x,y)\, dx\, dy\\
 &\leq
 \int_\Omega\,f(x) \cdot\left(\frac{1}{u^\gamma}-\Phi_{k}'(w)\right)
 T_\tau\left((u-w-\e)^+\right)\,dx\\
 &\leq
 \int_\Omega\,f(x) \cdot\left(\Phi_{k}'(u)-\Phi_{k}'(w)\right)
 T_\tau\left((u-w-\e)^+\right)\,dx \leq 0\,.
 \end{split}
\end{equation}
Here we used standard elliptic estimates and the fact that $H(x,y)$ is nonnegative by the definition of $g$ that is nondecreasing.
Furthermore we have that
\begin{equation}\nonumber
\begin{split}
 \int_{\mathbb{R}^{2N}} \frac{(|u(x) - u(y)|+|w(x) - w(y)|)^{p-2}\, (T_\tau\left((u-w-\e)^+(x)\right)-T_\tau\left((u-w-\e)^+(y)\right))^2\, }{|x - y|^{N+s\,p}}\, dx\, dy
\leq 0
 \end{split}
\end{equation}
thus proving that
\[
u\leq w+\e \leq v+\e\qquad \text{a.e. in}\quad \Omega
\]
and the thesis follows letting $\e\rightarrow 0$.
\end{proof}

\noindent
In light of Theorem \ref{comparison} we are now in position to conclude the proof of our main results.
\begin{proof}[Proof of  Theorem \ref{main}]
 If $u$ and $v$
are two solutions to \eqref{prob} with zero Dirichlet boundary condition, then we have that
$u\leq v$ by Theorem \ref{comparison}. In the same way it follows that $v\leq u$.
\end{proof}

\noindent
We now deduce a symmetry result from the uniqueness of the solution. We have the following:

\begin{proof}[Proof of Theorem \ref{Symmetry}]
By rotation and translation invariance, we may and we will assume that
$\Omega$ is symmetric in the $x_1$-direction and $f(x_1,x')=f(-x_1,x')$ (with $x'\in\R^{N-1}$).
Setting $v(x_1,x')\,:=\, u(-x_1,x')$ it follows that $v$ is a solution to \eqref{prob}  with zero 
Dirichlet boundary condition. By uniqueness, namely applying Theorem \ref{main}, it follows that $u=v$, that is $u(x_1,x')=u(-x_1,x')$ a.e., ending the proof.
\end{proof}

\end{document}